\documentclass[12pt]{article}

\usepackage[utf8]{inputenc}
\usepackage[T1]{fontenc}
\usepackage{amsthm}
\usepackage{eucal}
\usepackage{dsfont}
\usepackage{mathrsfs}
\usepackage{stmaryrd}
\usepackage{amsmath}
\usepackage{amsfonts}
\usepackage{fancyhdr}
\usepackage{amssymb}
\usepackage{xcolor}
\definecolor{Prune}{RGB}{99,0,60}
\usepackage{mdframed}
\usepackage{multirow} 
\usepackage{multicol} 
\usepackage{scrextend} 
\usepackage{tikz}
\usepackage{graphicx}
\usepackage[absolute]{textpos}
\usepackage{colortbl}
\usepackage{array}
\usepackage{geometry}
\usepackage[main=english]{babel}
\usepackage[final]{microtype}
\usepackage{lmodern}
\usepackage{textcomp}
\usepackage{pict2e}
\usepackage{enumitem}
\usepackage{hyperref}
\usepackage[nodayofweek]{datetime}
\usepackage{mathtools}
\mathtoolsset{showonlyrefs}

\hypersetup{            pdfpagemode={UseOutlines},            bookmarksopen,               pdfstartview={FitH},                  colorlinks,                     linkcolor={black},                 citecolor={black},                  urlcolor={black}                               }

\newdateformat{monthyear}{\monthname[\THEMONTH], \THEYEAR}

\newcommand{\N}{\ensuremath{\mathbb{N}}}
\newcommand{\R}{\ensuremath{\mathbb{R}}}

\theoremstyle{plain}
\newtheorem{theorem}{Theorem}[section]
\newtheorem{lemma}[theorem]{Lemma}
\newtheorem{prop}[theorem]{Proposition}
\newtheorem{corollary}[theorem]{Corollary}
\newtheorem{definition}[theorem]{Definition}
\theoremstyle{definition}

\usepackage{authblk}

\def\0{{\mathbf{0}}}

\renewcommand{\P}{\mathbb{P}}

\newcommand{\E}{\mathbb{E}}

\title{\textbf{About the regularity of degenerate non-local Kolmogorov operators under diffusive perturbations}}

\author[1]{\textbf{L. Marino}}
\author[2]{\textbf{S. Menozzi}}
\author[3]{\textbf{E. Priola}}

\affil[1]{\footnotesize{ Institute of Mathematics,
   Polish Academy
  of Sciences, Jana i J\c edrzeja \'Sniadeckich $8$, $00-656$, Warsaw, Poland.}}
\affil[2]{\footnotesize{ Laboratoire de Mod\'elisation Math\'ematique d'Evry (LaMME), UMR CNRS 8071,
 Universit\'e Paris-Saclay, Universit\'e d'Evry Val d'Essonne, $23$
Boulevard de France $91037$ Evry, France.}}
\affil[3]{\footnotesize Dipartimento di Matematica, Universit\`a di Pavia, Via Adolfo Ferrata $5$, $27100$ Pavia,
Italy.}

\begin{document}
\maketitle
\begin{abstract}
We study here the effects of a time-dependent second order perturbation to a degenerate Ornstein-Uhlenbeck type operator whose diffusive part can be either local or non-local. More precisely, we establish that some estimates, such as the Schauder and Sobolev ones, already known for the non-perturbed operator still hold, and with the same constants, when we perturb the Ornstein-Uhlenbeck operator with second order diffusions with coefficients only depending on time in a measurable way.
The aim of the current work is two-fold: we weaken the assumptions required on the perturbation in the local case which has been considered already in \cite{Marino:Menozzi:Priola22} and  we extend the approach presented therein to a wider class of degenerate Kolmogorov operators with non-local diffusive part of symmetric stable type.
\end{abstract}

{\small{\textbf{Keywords:} Degenerate  Ornstein-Uhlenbeck operators,  non-local parabolic equations, $L^p$ estimates, Schauder estimates, Poisson process.}}

{\small{\textbf{MSC:} Primary: $35$K$10$, $35$K$15$, $60$J$76$.}}

\setcounter{equation}{0}

\section{Introduction}

\subsection{General aims and scopes}
In this work we are interested in establishing estimates, namely Schauder and $L^p$-type, for some perturbation of degenerate Ornstein-Uhlenbeck operators, which satisfy a Kalman type condition (see assumption \textbf{[K]} below). For non degenerate diffusion operators it has been proved in \cite{Krylov:Priola17} that for some suitable second order perturbations, some well known estimates from the parabolic/elliptic theory remained true with the \textit{very} same constants. This also allowed the authors to prove therein that the Schauder and $L^p$ estimates for the heat equation actually did not depend on the dimension. The strategy therein relies on a probabilistic technique which consists in introducing a random source of Poisson type. Once averaged those discontinuities make appear a finite difference operator in the associated PDE. The final estimates then follow by some compactness arguments. The Poisson process actually allows very naturally to perform those  operations and up to now, to the best of our knowledge, this is the only approach which consents to derive such result. The question to obtain a purely analytic proof of the indicated stability results remains open.

In \cite{Marino:Menozzi:Priola22} we managed to prove that the strategy developed in \cite{Krylov:Priola17} was sufficiently robust to extend to degenerate Ornstein-Uhlenbeck operators satisfying a Kalman condition provided the second order spatial perturbation has coefficients which are continuous in time. We manage to get rid in the current work of this assumption and establish the aforementioned stability under {the mere time measurability and boundedness} for the perturbations. We also extend such approach to the non-local degenerate context. Namely, we consider degenerate stable type
Ornstein-Uhlenbeck operators and perturb these with the same kind of second order diffusion considered above.

\subsection{The model}
\label{Sec:Model}

Let us describe the general framework we are going to consider here. Fixed $N$ in $\N$, we denote by $\mathcal{S}(\R^N)$ the family of symmetric matrix in $\R^N\otimes \R^N$. We can then consider its subset $\mathcal{S}_0(\R^N)$, which stands for the closed convex cone of non-negative definite matrices, and its interior $\mathcal{S}_+(\R^N)$ which corresponds to the open subset of positive definite matrices.

Let us introduce two matrices $A,B$ in $\R^N\otimes \R^N$ such that $B$ belongs to $\mathcal{S}_0(\R^N)$ and satisfying the following Kalman rank condition:
\begin{trivlist}
\item[\textbf{[K]}] There exists a  non-negative integer $k$ such that
\begin{align} \label{kal}
{\rm Rank} [ B,AB,\ldots,
  A^{k} B
] =N,
\end{align}
where $ [B,AB,...,A^{k}B]$  is the $\R^N\otimes \R^{N(k+1)}$ matrix whose blocks are $B,AB,$ $\ldots , A^kB$.
 \end{trivlist}
 It has been shown (cf.\ \cite{Lanconelli:Polidoro94}) that, up to a possible change of variables, the space $\R^N$ can be decomposed as $\R^{d_0}\times \R^{d_1}$ (with $d_0\ge 1$) such that
\begin{equation}
\label{eq:def_B_0}
B  \, = \,  \begin{pmatrix}
 B_0  & 0_{d_0 \times d_1}\\
 0_{d_1 \times d_0} & 0_{d_1 \times d_1}
\end{pmatrix},
\end{equation}
for some matrix $B_0$ in $\mathcal{S}_+(\R^{d_0})$.
Moreover, if we denote by $\kappa_2>0$ the smallest eigenvalue of the matrix $B_0$, it then follows that
\begin{equation}
\label{eq:def_kappa_2}
B_0 \xi\cdot \xi \, \ge \, \kappa_2 |\xi|^2,
\end{equation}
for any $\xi$ in $\R^{d_0}$. Above, ``$\cdot$'' denotes the usual inner product in $\R^{d_0}$. From the non-degeneracy of $B_0$, the above condition \textbf{[K]} amounts to say that the vectors
\[
\{ e_1 , \ldots, e_{ d_0}, A e_1 ,
\ldots, Ae_{ d_0}, \ldots, A^k e_1 , \ldots, A^k
e_{ d_0}\} \;\;\; \mbox{generate} \;\; \R^N,
\]
where $\{e_i\colon i\in \{ 1,\cdots,d_0\}\}$ are the first $d_0$ vectors of  the canonical basis for $\R^N$. For simplicity, let us also denote by $\sigma$ in $\R^N\otimes \R^{d_0}$ the following matrix:
\[ \sigma  \, = \,  \begin{pmatrix}
 \sqrt{B_0} \\
 0_{d_1 \times d_0}
\end{pmatrix},\]
so that $\sigma\sigma^\ast=B$.

Fixed $\alpha$ in $(0,2]$, we can now introduce an operator $\mathcal{L}_\alpha$ on $C^\infty_b(\R^N)$, the space of all the bounded, smooth functions on $\R^N$ with bounded derivatives of any order, given by:
\begin{multline}
\label{eq:def_operator_cal_L}
\mathcal{L}_\alpha\phi(x) \, := \,
\mathds{1}_{\{\alpha =2\}} \text{Tr}(B D^2_x \phi(x)) + \\
\mathds{1}_{\{\alpha\neq 2\}}\int_{\R^{d_0}} \left[\phi(x+\sigma z)-\phi(x)-\langle D_x\phi(x),\sigma z\rangle\mathds{1}_{B(0,1)}(\sigma z)\right] \, \nu_\alpha(dz)
\end{multline}
where $\nu_\alpha$ is a symmetric, non-degenerate $\alpha$-stable measure on $\R^{d_0}$, i.e.
\begin{equation}
\label{eq:def_measure_nu_alpha}
\nu_\alpha(C) \, = \, \int_0^\infty\int_{\mathbb{S}^{{d_0}-1}}\mathds{1}_{C}(r\theta)\mu(d\theta)\frac{dr}{r^{1+\alpha}},\quad C\in\mathcal{B}(\R^{d_0}),
\end{equation}
for a symmetric, finite measure $\mu$ on $\mathbb{S}^{{d_0}-1}$ for which the following \emph{non-degeneracy} condition holds:
\begin{trivlist}
\item[\textbf{[ND]}] there exists $\kappa_\alpha>0 $ such that
\[\int_{\mathbb{S}^{{d_0}-1}}|\lambda\cdot \theta|^\alpha \, \mu(d\theta) \, \ge \, \kappa_\alpha |\lambda|^\alpha, \quad \lambda \in \R^{d_0}.\]
\end{trivlist}
In \eqref{eq:def_operator_cal_L}, $\langle \cdot , \cdot \rangle$ denotes the usual inner product in $\R^N$ and $D^2_x\phi$ represents the full Hessian matrix in $\R^N\otimes\R^N$ with respect to $x$.
From now on, we will say that assumption [\textbf{A}] holds when the above conditions [\textbf{K}] and [\textbf{ND}] are in force.

We will use, as an underlying \textit{proxy} operator, a degenerate Ornstein-Uhlenbeck operator $\mathcal{L}^{\text{ou}}_\alpha$ of the form
\begin{equation}
\label{DEF_OU_OP_PROXY}
\mathcal{L}^{\text{ou}}_\alpha \phi(x) \, := \, \mathcal{L}_\alpha\phi(x) + \langle A x , D_x \phi(x)\rangle,
\end{equation}
for $x$ in $\R^N$.  In the Gaussian case (i.e.\ when $\alpha=2$), assumption \textbf{[K]} (which also often appears in control theory; see e.g.\ \cite{book:Zabczyk95}) is equivalent to the  H\"ormander condition on the
commutators (c.f.\ \cite{Hormander67}) ensuring the hypoellipticity of the operator $\partial_t-\mathcal{L}^{\text{ou}}_\alpha$. In particular, it implies the existence and the smoothness of a distributional solution to the following equation:
\begin{equation}
\label{eq:OU_initial:intro}
\begin{cases}
    \partial_t u(t,x) \, = \,  \mathcal{L}^{\text{ou}}_\alpha u(t,x) +f(t,x) &\mbox{ on }\R^N_T;\\
   u(0,x)\, = \, 0 &\mbox{ on } \R^N,
\end{cases}
\end{equation}
where $\R_T^N:=(0,T)\times \R^N$ for some fixed final time $T>0$ and $f$ is a function in $C_c^\infty(\R^N_T)$, the set of all the smooth functions on $\R^N_T$ with compact support.

From this point further, we assume to have fixed a time dependent matrix $S\colon (0,T)\to \mathcal{S}_0(\R^N)$ which is {bounded measurable} in $(0,T)$. We are then interested in the following perturbation of $\mathcal{L}^{\text{ou}}_\alpha$:
\begin{equation}
\label{eq:def_pert_operator}
\mathcal{L}_{\alpha,t}^{\text{pert}} \phi(x) \, :=\, \mathcal{L}^{\text{ou}}_\alpha \phi(x)+ {\rm Tr}(S(t) D^2_x \phi(x)).
\end{equation}
Our main interest here is to understand how such a perturbation influences some well-known estimates for the Ornstein-Uhlenbeck operator $\mathcal{L}^{\text{ou}}_\alpha$. For example, when considering the diffusive setting $\alpha=2$, Theorem $3$ in \cite{Bramanti:Cupini:Lanconelli:Priola10} (but see also \cite[Section $2.3$]{Marino:Menozzi:Priola22}), showed that for any fixed $p$ in $(1,+\infty)$, there exists $c_p:=c_p(T,A, B, d_0,d_1)$  such that the solution $u$ to Cauchy problem \eqref{eq:OU_initial:intro} satisfies
\begin{equation}\label{de2}
\| D^2_{x_0} u \|_{L^p (\R^N_T)}\, \le \,   c_p  \|f\|_{L^p (\R^N_T)},
\end{equation}
where for any $(t,x)=(t,x_0,x_1)$ in $\R^N_T=(0,T)\times\R^{d_0}\times\R^{d_1}$, $D^2_{x_0}u(t,x)$ stands for the Hessian matrix in $\R^{d_0}\otimes\R^{d_0}$ with respect to the variable $x_0$ and $L^p(\R^N_T)$ is the standard $L^p$-space with respect to the Lebesgue measure. Under some additional regularity assumptions, namely the continuity in time for $S(t)$, we showed in \cite{Marino:Menozzi:Priola22} that the above estimates are indeed stable under a diffusive perturbation ${\rm Tr}(S(t)D^2_x)$. The main objective of the present paper is to extend the previous results in \cite{Marino:Menozzi:Priola22} to possibly fractional proxy operators and to diffusive perturbations where $S(t)$  {is only bounded measurable}. We will also show some applications of the previous arguments to elliptic or more general parabolic equations.

The article is organised as follows. We introduce in Section \ref{SEC_NOT} the notion of solution considered as well as some useful notations about the \textit{natural} function spaces associated with the specific degenerate operators described above. In that section, we also establish a maximum principle for a large class of operators, see Theorem \ref{thm:max_princ}, which we think has independent interest. Section \ref{ESTENSIONI} is then devoted to the statement of our main results, see Theorem \ref{thm:main_theorem}, and some related extensions (parabolic operators with potential and elliptic counterpart). Section \ref{SEC_PROOF} is eventually dedicated to the proof of the main result. As we mentioned before, the approach therein is based on the Poisson perturbative approach firstly considered in \cite{Krylov:Priola17}.

\setcounter{equation}{0}
\section{Useful notations}
\label{SEC_NOT}
\subsection{Definition of solution}

Let us consider in this section a slightly more general framework. We are interested in the following Cauchy problem:
\begin{equation}
\label{eq:Cauchy_problem_gen0}
\begin{cases}
\partial_t v(t,x)   =  L_{t}v(t,x)  + f(t,x) &\mbox{ on }\R^N_T; \\
v(0,x) \, = \, 0 &\mbox{ on }\R^N.
\end{cases}
\end{equation}
Above, the operator $L_{t}$ on $C^\infty_b(\R^N)$ is given by
\begin{equation}
\label{eq:def_full_operator}
\begin{split}
L_t\phi(x) \, &:= \, \text{Tr}\left(Q(t)D^2_x \phi(x)\right)+\langle  b(t,x), D_x\phi(x) \rangle-c(t)\phi(x)+\mathcal N_t \phi(x), \\
\mathcal N_t \phi(x) \,&:= \, \int_{\R^{d_0}} \left[\phi(x+ \gamma(t)z)-\phi(x)-\langle D_x\phi(x), \gamma(t)z\rangle\mathds{1}_{B(0,1)}(z)\right] \, \nu_\alpha(dz),
\end{split}
\end{equation}
where $Q\colon (0,T)\to \mathcal{S}_0(\R^N)$, $b\colon (0,T) \times \R^N \to \R^N$, $c\colon (0,T) \to [0,+\infty)$ and $\gamma\colon (0,T)\to \R^N\otimes \R^d$ are Borel measurable and $\nu_\alpha$ is the $\alpha$-stable L\'evy measure given in \eqref{eq:def_measure_nu_alpha}.

We will denote by $B_b\left(0,T;C^\infty_c(\R^N)\right)$ the space of all Borel bounded functions $\phi\colon \R^N_T \to \R$ such that $\phi(t,\cdot)$ is smooth and compactly supported for any $t$ in $(0,T)$, for any $n$ in $\N$ the $C^n(\R^N)$-norms of $\phi(t,\cdot)$ are bounded in time and the supports of the functions $\phi(t,\cdot)$ are contained in the same ball. We remark that the choice of such space, which turns out to be rather natural when considering sources that are possibly discontinuous in time, is strictly related to the proof technique used in \cite{Krylov:Priola17} and based on the Poisson process (cf.\ proof of Lemma \ref{lemma:finite_diff}).

For a bounded, Borel measurable function $f\colon \R^N_T\to \R$, we are going to interpret Cauchy Problem \eqref{eq:Cauchy_problem_gen0} in  an {\it integral in time} form:
\begin{equation}
\label{int2}
u(t,x) \, = \, \int_0^t \left[f(s,x)+ L_su(s,x) \right]\, ds.
\end{equation}
\begin{definition}[Integral solution for the Cauchy problem]\label{INT_SOL}
We say that a function $u\colon [0,T]\times \R^N \to \R$ is a solution to Equation \eqref{eq:Cauchy_problem_gen0} if $D^{\zeta}_x u$ is bounded continuous on $[0,T]\times \R^N$ for any $\zeta$ in $\N^N$ such that $|\zeta|\le 2$ and \eqref{int2} holds for any $(t,x)$ in $[0,T]\times\R^N$.
\end{definition}
We recall that for a given multi-index $\zeta$ in $\N^N$, $D^{\zeta}_x u$ denotes the iterated spatial derivatives of $u$ of order $\zeta$, i.e.
\[D^{\zeta}_x u \,:=\,D^{\zeta_1}_{x_1}D^{\zeta_2}_{x_2}\dots D^{\zeta_n}_{x_n}u.\]

Under some additional assumptions (see Theorem \ref{thm:max_princ} below), it is indeed possible to show that, when it exists,  an integral in time solution  $u$ to Cauchy problem \eqref{eq:Cauchy_problem_gen0} is unique and the following maximum principle holds:
\begin{equation}
\label{eq:Max_princ}
\sup_{(t,x)\in[0,T]\times \R^N}|u(t,x)| \, \le \, T\sup_{(t,x)\in\R^N_T}|f(t,x)|.
\end{equation}
While such estimates are well-known in the diffusive setting (see e.g.\ \cite[Theorem $4.1$]{Krylov:Priola10} or \cite{Marino:Menozzi:Priola22}), we could not find a precise proof in the fractional case. For this reason, we now present it for future reference.

\begin{theorem}[Maximum principle]
\label{thm:max_princ}
Let $\alpha\in (0,2)$. Assume in addition that $Q\colon (0,T)\to \mathcal{S}_0(\R^N)$ and $\gamma\colon(0,T)\to \R^N\otimes\R^d$ are bounded. Let $u$ be a solution to Cauchy problem \eqref{eq:Cauchy_problem_gen0} in the sense of Definition \ref{INT_SOL}, with $c=0$ and $b=0$. Then, the maximum principle \eqref{eq:Max_princ} holds for $u$.
\end{theorem}
\begin{proof}
By considering $u$ and $-u$, it is clearly enough to show that
\[
\sup_{(t,x)\in\R^N_T}u(t,x) \, \le \, T \sup_{(t,x)\in\R^N_T}|f(t,x)|.
\]
In order to mimic the proof technique in \cite{Krylov:Priola10} (see Theorem 4.1 therein in the diffusive setting), we firstly notice that the function
$\tilde{u}\colon [-T,0]\times\R^N\to \R$ given by
\[\tilde u(t,x) \, := \, e^t\left(u(-t,x) + t \|f\|_{\infty}\right),\]
solves the following Cauchy problem:
\begin{equation}
\label{eq:Cauchy_problem_gen1}
\begin{cases}
\partial_t \tilde{u}(t,x)+\tilde{L}_t\tilde{u}(t,x) \, = \, \tilde{f}(t,x) &\mbox{ on }(-T,0)\times \R^N; \\
\tilde{u}(0,x) \, = \, 0 &\mbox{ on }\R^N.
\end{cases}
\end{equation}
where $\tilde{f}(t,x):=e^t(\|f\|_{\infty}-f(-t,x))\ge 0$ and $\tilde{L}_t$ is the operator given in \eqref{eq:def_full_operator} with coefficients $\tilde{Q}(t)=Q(-t)$, $\tilde{c}(t)=1$ and $\tilde{\gamma}(t)=\gamma(-t)$. For notational convenience, let us also denote $\tilde{\mathcal{N}}_t:=\mathcal{N}_{-t}$.
In order to conclude the proof, it is then enough to show that
\[\sup_{(t,x)\in [-T,0]\times\R^N}\tilde{u}(t,x)\, \le\, 0.\]

Let us introduce now the following barrier function $w(t,x)=\exp(-Ct)\ell_\beta(x)$ where $ \ell_\beta(x)=(1+|x|^2)^{\frac \beta 2}$ for some $\beta<\alpha\wedge 1$ and $C>0$.
One can observe that $\ell_\beta $ is smooth and that there exists $\bar C\ge 1 $ such that for any multi-index $\zeta$ in $\N^n$ such that $|\zeta|\le 2$,
\begin{align}\label{DER_BAR}
|D_x^\zeta \ell_\beta(x)|
\le \bar C(1+|x|^2)^{\frac \beta2-\frac{|\zeta|}2}.
\end{align}
In particular, by our choice of $\beta$, $D_x\ell_\beta$ is bounded. Moreover, it is possible to choose $C>0$ large enough so that
\begin{align}
\label{NEG_BARRIER}
(\partial_t+\tilde{L}_t)w(t,x) \,\le  \, 0.
\end{align}
Indeed, one can notice that $\partial_t w(t,x) = -Ce^{-Ct}\ell_\beta(x)$ and
\[
|\text{Tr}\left(\tilde{Q}(t)D^2_x w(t,x)\right)|\, \le \, e^{-Ct}\bar C.
\]
Furthermore, one can infer from the $\beta$-H\"older regularity of $\ell_\beta$ that
\[
|\tilde{\mathcal{N}}_tw(t,x)|\, = \, \left|e^{-Ct}{\rm p.v.}\int_{\R^d}[\ell_\beta(x+\gamma(t) z)-\ell_\beta(x)]\,  \nu_\alpha( dz)\right|\le \bar Ce^{-Ct}.
\]
Now, the point is to prove that for any fixed $\eta>0$ and any $(t,x)\in [-T,0]\times \R^N$,
\begin{equation}\label{EST_WITH_POTENTIAL_WITH_BARRIER}
v(t,x)\, := \, e^{-t/2}(\tilde{u}(t,x)-\eta w(t,x)) \, \le \, 0,
\end{equation}
where $w(t,x)$ is the previous barrier which satisfies \eqref{NEG_BARRIER}. The statement indeed follows for $\tilde{u}$ letting $\eta$ go to zero and using that $u$ is bounded in $\R^N_T$.

Since $\tilde{u}$ is bounded and $w$ goes to infinity with $|x| $, we deduce that the supremum of $v$ in $[-T,0]\times \R^N$ has to be attained at some point $(t_0,x_0) \in [-T,0]\times \R^N$. We can assume without loss of generality that $t_0$ is in $[-T,0)$. Moreover, since $v$ is smooth in space, $D_x v(t_0,x_0)=0$ and $D^2_xv(t_0,x_0) $ is non-positive definite. Then, by the space-time continuity of $v$ and its spatial derivatives, we notice that for any $\tau>0$, there exists  $\theta:=\theta(\tau) $ such that for any $(t,z)\in [-T,0)\times \R^N$ such that $|t-t_0|\vee|z-x_0|<\theta$,
\begin{equation}
\label{eq:proof_PM_ctny}
v(t,x_0)\ge 0, \qquad
|D_xv(t,z)| \le \tau,\qquad  [D^2_x v(t,z)]_{i,j} \le \tau\delta_{i,j}.
\end{equation}
We can now write for any $0\le t-t_0\le \theta$,
\[0 \, \ge \, v(t,x_0)- v(t_0,x_0)\, =\, \int_{t_0}^{t} \partial_s v(s,x_0) ds. \]
We then observe that for any $(s,x)$ in $[-T,0]\times \R^N$, we have that
\begin{align}
\label{eq:proof_max}
\partial_s v (s,x)\, &= \, -\frac 12 v(s,x) +e^{-s/2}(\partial_s \tilde{u}(s,x)-\eta\partial_sw(s,x))\\ \notag
&= \, -\frac 12 v(s,x)-\tilde{L}_s v(s,x) + e^{-s/2}\tilde{f}(s,x)-\eta e^{-s/2}\left(\partial_s+\tilde{L}_s\right)w(s,x) \\
&\ge\,-\frac 12 v(s,x)-\tilde{L}_s v(s,x),\notag
\end{align}
exploiting, in the last step, \eqref{NEG_BARRIER}. Hence, by the definition of the operator $\tilde{L}_s$ in \eqref{eq:def_full_operator} (with the choice of coefficients described after \eqref{eq:Cauchy_problem_gen1}), it holds that
\begin{equation}
\label{eq:proof_max0}
\int_{t_0}^tv(s,x_0) \,  ds \,\le \, 2\int_{t_0}^t \left[{\rm Tr}\left(\tilde{Q}(s) D^2_xv(s,x_0)\right) +\tilde{\mathcal{N}_s} v(s,x_0)\right]\, ds.
\end{equation}
Taking $t-t_0<\theta$ and $s$ in $[t,t_0]$, one can easily conclude from \eqref{eq:proof_PM_ctny} that
\begin{equation}
\label{eq:proof_max1}
{\rm Tr}\left(\tilde{Q}(s)D^2_xv(s,x_0)\right) \,  \le  \, c\tau.
\end{equation}
On the other hand, we can split $\tilde{\mathcal{N}_s} v$ into $I^1_sv+I^2_sv$, where
\begin{align*}
I^1_sv(s,x_0)\, &:= \, \int_{|z|\le t-t_0}\left[ v(s,x_0+\tilde{\gamma}(s)z)-v(s,x_0)-\langle \nabla  v(
s,x_0),\gamma(s)z\rangle \right] \nu_\alpha(dz),\\
I^2_sv(s,x_0)\, &:= \, \int_{|z|> t-t_0}\left[ v(s,x_0+z)-v(s,x_0) \right] \nu_\alpha(dz).
\end{align*}
Using again \eqref{eq:proof_PM_ctny}, one can then derive that
\begin{equation}
\begin{split}
\label{eq:proof_max2}
I^1_sv(s,x_0) \, &= \, \int_0^1\int_{|z|\le t-t_0} \langle D^2_x v(s,x_0+\lambda \gamma(s)z)\gamma(s)z,\gamma(s)z\rangle \, \nu_\alpha(dz)d\lambda \\
&\le \, c\tau \int_{|z|\le 1} |z|^2 \, \nu_\alpha(dz) \, \le  \,c\tau.
\end{split}
\end{equation}
On the other hand, since $v(t_0,x_0)$ is a maximum, so that $v(t_0,x_0+z)- v(t_0,x_0)\le 0$, we also have that
\[\begin{split}
I^2_s v&(s,x_0) \\
&= \, \int_{|z|> t-t_0} \left[v(s,x_0+\gamma(s)z)\pm v(t_0,x_0+\gamma(s)z)\pm v(t_0,x_0)-v(s,x_0)\right] \, \nu_\alpha(dz)\\
&\le \, \int_{|z|>t-t_0} [v(s,x_0+\gamma(s)z)-v(t_0,x_0+\gamma(s)z)+v(t_0,x_0)-v(s,x_0)] \, \nu_\alpha(dz)\\
&= \, \int_{t_0}^{s}\int_{|z|> t-t_0} \left[\partial_rv(r,x_0+\gamma(s)z)-\partial_r v(r,x_0)\right] \, \nu_\alpha(dz)dr.
\end{split}\]
Similarly to \eqref{eq:proof_max}, we note that
\[\partial_s v (s,x)\, = \, e^{-s/2}\left[\eta \tilde{C}w(s,x)+\tilde{f}(s,x)-\tilde{L}_s\tilde{u}(s,x))-\frac 12 \tilde{u}(s,x)\right],\]
where the constant $\tilde{C} := C+1/2$ and $C$ defined in \eqref{NEG_BARRIER}. Hence,
\begin{multline*}
I^2_s v(s,x_0) \, \le \, \int_{t_0}^{s}\int_{|z|>t-t_0}
\Bigl|\left(\eta \tilde{C}w+\tilde{f}-\tilde{L}_r\tilde{u}-\frac 12 \tilde{u}\right)(r,x_0+\gamma(s)z)\\
-\left(\eta \tilde{C}w+\tilde{f}-\tilde{L}_r\tilde{u}-\frac 12 \tilde{u}\right)(r,x_0)\Bigr|\, \nu_\alpha(dz)dr
\end{multline*}
When $\alpha$ is in $(0,1)$, we can use the boundedness of $\tilde{u}$ and $\tilde{f}$ and the $\beta$-H\"older continuity of $\ell_\beta$ to infer that
\begin{align}
\label{eq:proof_max3}
I^2_s v(s,x_0) \, &\le \, c(s-t_0)\int_{|z|\ge t-t_0}\left(1+|z|^\beta\right) \nu_\alpha(dz) \, \le \, c(s-t_0)\int_{t-t_0}^{+\infty} r^{-(1+\alpha)}(1+r^\beta) \, dr \notag\\
&\le \, c(s-t_0)(t-t_0)^{-\alpha}.
\end{align}
If instead $\alpha\in (1,2)$, we need to apply an additional Taylor expansion:
\begin{multline*}
 I^2_s v(s,x_0) \, \le \, \int_0^1\int_{t_0}^{s}\int_{|z|\ge t-t_0}\Bigl| \left\langle \left(\eta \tilde{C}D_xw+D_x\tilde{f}\right)(r,x_0+\lambda \gamma(s)z),\gamma(s)z\right\rangle \\
 -\left\langle \left(D_x\left(\tilde{L}_r\tilde{u}\right)+\frac 12 D_x\tilde{u}\right)(r,x_0+\lambda \gamma(s)z),\gamma(s)z\right\rangle\Bigr| \,  \nu_\alpha(dz)drd\lambda.
\end{multline*}
Using again the boundedness of the functions involved, i.e. $D_x\tilde{u}$, $D_xw$ and $D_x\tilde f $, and noting that $D_x\left( \tilde{L}_r \tilde{u}\right)= \tilde{L}_r\left(D_x \tilde{u}\right)$, we can conclude in the same way as above that
\begin{equation}
\label{eq:proof_max4}
I^2_s v(s,x_0) \, \le  \, c(s-t_0)(t-t_0)^{1-\alpha}.
\end{equation}
Applying estimates \eqref{eq:proof_max1}, \eqref{eq:proof_max2} and \eqref{eq:proof_max3} or \eqref{eq:proof_max4} inside \eqref{eq:proof_max0}, we finally get that
\[
\int_{t_0}^tv(s,x_0) \,  ds \,\le \,  c\left[\tau +(t-t_0)^\varepsilon\right] (t-t_0),
\]
for some  $\varepsilon>0$. Using the continuity of $v$, we can divide both sides by $t-t_0$ and let $t$ go to $t_0$ so that
\[v(t_0,x_0) \, \le\, c\tau.\]
Since $\tau>0$ is arbitrary, we have concluded the proof.
\end{proof}

\subsection{Definition of the anisotropic norms}

We firstly recall from \cite{Lanconelli:Polidoro94} that the state space $\R^N$ can be split into $\R^{d_0}\times\R^{d_1}$ so that, under such decomposition, the matrix $B$ assumes the form in \eqref{eq:def_B_0}.
More precisely, in the aforementioned work, it was established that
Assumption \textbf{[K]} is equivalent to the fact that there exists $k\in \N $ and positive integers $\{\mathfrak d_i\colon i\in\{1,\cdots, k\}\}$ such that $\sum_{i=1}^k \mathfrak d_i=d_1 $ and for all $i\in \{1,\cdots,k\} $, setting $\mathfrak d_0=d_0 $ and $\sum_{m=0}^{-1} =0$, the matrices
$${\mathscr A}^i\, := \, (A_{j,\ell})_{(j,\ell)\in \{\sum_{m=0}^{i-1}\mathfrak d_m+1,\cdots,\sum_{m=0}^{i}\mathfrak d_m\}\times \{\sum_{m=1}^{i-1}\mathfrak d_m +1,\cdots,\sum_{m=1}^{i}\mathfrak d_m\}},$$
have rank $\mathfrak d_i $.  Moreover, the matrix $A$ writes:
\begin{equation} \label{sotto}
A \, = \,
    \begin{pmatrix}
        \ast   & \ast  & \dots  & \dots  & \ast   \\
         {\mathscr A}^1  & \ast  & \ddots & \ddots  & \vdots   \\
        0_{\mathfrak d_2,d_0}      & {\mathscr A}^2   & \ast  & \ddots & \vdots \\
        \vdots &\ddots & \ddots& \ddots & \ast \\
        0_{\mathfrak d_k,d_0}      & \dots & 0_{\mathfrak d_k,\mathfrak d_{k-1}}     & {\mathscr A}^{k}    & \ast
    \end{pmatrix}.
\end{equation}
We can then write $x\in \R^N $ as $x=(x_0,x_1,\cdots, x_k) $ with $x_i\in \R^{\mathfrak{d}_i},\ i\in \{0,\cdots,k \}$.
In order to properly introduce the anisotropic functional space we are going to consider, we follow \cite{Huang:Menozzi:Priola19} by introducing the orthogonal projection $p_i\colon \R^N\to \R^{\mathfrak{d}_i}$ such that $p_i(x)=x_i$ and denoting its adjoint by $E_i\colon \R^{\mathfrak d_i}\to \R^N$, for any $i$ in $\llbracket 0, k \rrbracket$. For notational simplicity, let us denote
\begin{equation}\label{INDEXES}
\alpha_{i} \,:=\, \frac{\alpha}{2}\frac{1}{1+\alpha i}.
\end{equation}
The threshold in \eqref{INDEXES} might seem awkward at first sight. While the first term $\alpha/2$ relates to the maximal regularity associated with the fractional Laplacian $\Delta^\alpha$, the second one actually corresponds to the index needed to get
the invariance by dilations for the harmonic functions associated with the principal part of the Ornstein-Uhlenbeck operator $\mathcal{L}^{\text{ou}}_{\alpha}$. Namely, one considers the operator $\mathcal{L}_{0,\alpha}^\text{ou}$ given in \eqref{DEF_OU_OP_PROXY} with respect to the matrix
\begin{equation}
\label{eq:def_A_0}
A_0 \, = \,
    \begin{pmatrix}
        0_{\mathfrak d_0, \mathfrak d_0}   & 0_{\mathfrak d_0, \mathfrak d_1}  & \dots  & \dots  & 0_{\mathfrak d_0, \mathfrak d_k}   \\
         {\mathscr A}^1  & 0_{\mathfrak d_1,\mathfrak d_1}  & \ddots & \ddots  & \vdots   \\
        0_{\mathfrak d_2,\mathfrak d_0}      & {\mathscr A}^2   & 0_{\mathfrak d_2,\mathfrak d_2}  & \ddots & \vdots \\
        \vdots &\ddots & \ddots& \ddots & \vdots \\
        0_{\mathfrak d_k,\mathfrak d_0}      & \dots & 0_{\mathfrak d_k,\mathfrak d_{k-1}}     & {\mathscr A}^{k}    & 0_{\mathfrak d_k,\mathfrak d_k}
    \end{pmatrix}.
\end{equation}
Note that $A_0,B$ again satisfy \textbf{[K]}. If $ (\partial_t -\mathcal{L}_{0,\alpha}^\text{ou})u(t,x)=0$ then for all $ \lambda>0$ $(\partial_t -\mathcal{L}_{0,\alpha}^\text{ou})u\big(\delta_\lambda(t,x) \big)=0$ where the dilation operator
\[\delta_\lambda(t,x)=(\lambda^{1/\alpha}t , \lambda x_0, \lambda^{1/(1+\alpha)} x_1,\cdots, \lambda^{1/(1+\alpha k)} x_k),\]
precisely exhibits the exponents in \eqref{INDEXES} for the spacial components.

Since the main focus of the present work is on Schauder and Sobolev type estimates for solutions to Cauchy problem \eqref{eq:OU_initial:intro},
we now briefly recall the definition of Sobolev and H\"older norms in our anisotropic context. Fixed $\beta$ in $(0,1)$, let us denote by $\Delta^\beta_{x_i}$ the $\beta$-fractional Laplacian  along the $i$-th direction, i.e.
\[\Delta^{\beta}_{x_i}\phi(x) \, := \, \text{p.v.}\int_{\R^{\mathfrak d_i}}\left[\phi(x+E_i z)-\phi(x)\right] \frac{dz}{|z|^{\mathfrak d_i+ 2\beta}}, \quad x \in \R^N,\]
for any smooth enough function $\phi \colon \R^N\to \R$.
Given $p$ in $[1,+\infty)$, we can now define the \textit{homogeneous} Sobolev space $\dot{W}^{2,p}_d(\R^N_T)$ as the family of all the functions $\phi\colon \R^N_T\to \R$ in $L^p(\R^N_T)$ such that for any $i$ in $\llbracket 0, k \rrbracket$, $\Delta^{\alpha_i}_{x_i}\phi(t,x)$ is well defined for almost every $(t,x)$ in $\R^N_T$ and
\[\Delta^{\alpha_i}_{x_i}\phi(t,x) \, := \,  \Delta^{\alpha_i}_{x_i}\phi(t,\cdot)(x)   \text{ belongs to } L^p(\R^N_T).\]
It is endowed with the natural \textit{semi}-norm $[ \cdot ]_{\dot W_d^{\alpha,p}(\R^N_T)}$ given by:
 \begin{equation}\label{SEMI_NORM_SOB}
 [\phi ]_{\dot W_d^{\alpha,p}(\R^N_T)}^p \, = \, \sum_{i=0}^k\Vert \Delta^{\alpha_i}_{x_i}\phi \Vert_{L^p(\R^N_T)}^p.
\end{equation}

Following Krylov \cite{book:Krylov96}, for  some fixed $\ell$ in $\N_0:=\N\cup\{0\}$ and $\beta$ in $(0,1]$, we introduce for a function $\phi\colon \R^N\to \R$ the Zygmund-H\"older semi-norm  as
\[[\phi]_{C^{\ell+\beta}} \, := \,
\begin{cases}
\sup_{\vert \vartheta \vert= \ell}\sup_{x\neq y}\frac{\vert D^\vartheta\phi(x)-D^\vartheta\phi(y)\vert}{\vert x-y\vert^\beta} , & \mbox{if }\beta \neq 1; \\
\sup_{\vert \vartheta \vert= \ell}\sup_{x\neq y}\frac{\bigl{\vert}D^\vartheta\phi(x)+D^\vartheta\phi(y)-2D^\vartheta\phi(\frac{x+y}{2}) \bigr{\vert}}{\vert x-y \vert}, & \mbox{if } \beta =1,
\end{cases}\]
(we are using usual multi-indices $\vartheta$ for the partial derivatives).
Consequently, the Zygmund-H\"older space $C^{\ell+\beta}_b(\R^N)$ is the family of bounded functions $\phi\colon \R^N
\to\R$ such that $\phi$ and its derivatives up to order $\ell$ are continuous and the norm
\[\Vert \phi \Vert_{C^{\ell+\beta}_b} \,:=\,  \sum_{i=0}^{\ell}\sup_{\vert\vartheta\vert = i}\Vert D^\vartheta\phi
\Vert_{\infty}+[\phi]_{C^{\ell+\beta}} \,\text{ is finite.}
\]
We can now define  the anisotropic Zygmund-H\"older spaces associated with the current setting and which again reflect the various scales already introduced in \eqref{INDEXES}. Let $\gamma\in (0,3)$, the space $C^{\gamma}_{b,d}(\R^N)$ is
the family of functions $\phi\colon \R^N\to \R$ such that for any $i$ in $\llbracket 0,k\rrbracket$ and any $x_0$ in $\R^N$, the real  function
\[y\in  \R^{\mathfrak d_i}\,  \to \, \phi(x_0+E_i(y))  \,\text{ belongs to }C^{\gamma/(1+\alpha i)}_b\left(\R^{\mathfrak d_i}\right),\]
with a norm bounded by a constant independent from $x_0$.

We will also consider  the corresponding natural semi-norm for the related homogeneous space
\begin{equation}\label{eq:def_anistotropic_norm}
[\phi]_{C^{\gamma}_d} \,:=\,\sum_{i=0}^{k}\sup_{x_0\in \R^N} \left[\phi\left(x_0+ E_i(\cdot)\right)\right]_{C^{\gamma/(1+\alpha i)}(\R^{\mathfrak d_i})}.
\end{equation}

We finally remark that we have denoted by $C^{\gamma}_{d}$ and $\dot{W}^{2,p}_{d}$ the anisotropic functional spaces because the regularity exponents reflect again the multi-scale features of the system. In particular, the H\"older norm could equivalently be defined through the corresponding spatial parabolic distance $d$ defined as follows: For any $x,x'\in \R^N$:
\[d(x,x')\, := \, \sum_{i=0}^{k}\vert x_i-x_i'\vert^{\frac{1}{1+\alpha i}},\]
where the exponents are again those who appeared in \eqref{INDEXES}.

\setcounter{equation}{0}
\section{Main results}
\label{ESTENSIONI}
Let  $f$ be in $B_b\left(0,T;C^\infty_c(\R^N)\right)$ and $u$ the unique solution  to the corresponding Cauchy problem \eqref{eq:OU_initial:intro}.

In \cite{Huang:Menozzi:Priola19}, see also \cite{Chen:Zhang19}  and \cite{Menozzi18} where time inhomogeneous coefficients are considered as well, it has been proven that if $A,B $ satisfy \textbf{[K]} and the diagonal and the strictly upper diagonal elements of $A$ in \eqref{sotto} are equal to zero (i.e., $A = A_0$ in \eqref{eq:def_A_0}) then the following Sobolev estimates hold (with $p\in (1,+\infty)$):
\begin{equation}\label{eq:Sobolev_estim}
[ u ]_{\dot W^{\alpha,p}_d(\R^N_T)} \, \le \, C_p\Vert f \Vert_{L^p(\R^N_T)},
\end{equation}
for some positive constant $C_p := C(p,\alpha,A,B, d_0, d_1)$. We remark that the specific structure assumed on $A$ is actually due to the fact that for such matrices there is an underlying homogeneous space structure which makes easier to establish maximal regularity estimates (see e.g. \cite{book:Coifman:Weiss} in this general setting). If  $A,B $ satisfy \textbf{[K]} with  a general $A$ as in \eqref{sotto}, having non-zero strictly upper diagonal entries, we believe that  the approach in \cite{Bramanti:Cupini:Lanconelli:Priola10} could extend to show that  \eqref{eq:Sobolev_estim} still holds in this general setting.
In particular, non-zero entries in the diagonal do  not create additional difficulties but they introduce a dependence in the final time $T$ in the previous estimates.
However, the estimates in this more general framework have not been, up to our best knowledge, proven yet.

Let us now fix $\beta$ in $(0,1)$ such that $\beta<\alpha$. Under Kalman condition \textbf{[K]}, Lemmas $9$ and $10$ in \cite{Marino20} (but see also \cite{Chaudru:Honore:Menozzi18_Sharp, Lunardi97})  show that the solution $u$ also verifies the following anisotropic Schauder estimates:
\begin{equation}\label{eq:Schauder_estim}
\sup_{t\in [0,T]}[ u(t,\cdot) ]_{C^{\alpha+\beta}_d} \, \le \,  C_\beta \sup_{t\in(0,T)} [f(t,\cdot)]_{C^{\beta}_d},
\end{equation}
for some positive constant $C_\beta:=C(\beta,\alpha,A,B,d_0,d_1,T)$. Moreover, it is possible to show (see again \cite{Marino20} in a fractional setting) that, when considering a dilation invariant matrix $A$, i.e.\ when $A=A_0$ in \eqref{eq:def_A_0}, the constants $c_p,C_\beta$ appearing in the estimates \eqref{de2}-\eqref{eq:Schauder_estim} are indeed independent from the final time $T$ (see also \cite{Chaudru:Menozzi:Priola20}).

Our main result shows that indeed the above estimates are invariant under time dependent diffusive perturbations.

\begin{theorem}
    \label{thm:main_theorem}
Under [\textbf{A}], let $f$ be in $B_b(0,T;C^\infty_c(\R^N))$. Then, there exists a unique solution $u$ to the perturbed Cauchy problem
\begin{equation}
\label{eq:OU_initial:intro_pert}
\begin{cases}
    \partial_t u(t,x) \, = \,  \mathcal{L}^{\text{pert}}_{\alpha,t} u(t,x) +f(t,x)\quad &\mbox{ on }\R^N_T;\\
   u(0,x)\, = \, 0 \quad &\mbox{ on } \R^N,
\end{cases}
\end{equation}
where $\mathcal{L}^{\text{pert}}_{\alpha,t} $ is defined in \eqref{eq:def_pert_operator} for a {bounded measurable} perturbation $S(t)$. Moreover, for any $p>1$ and any $\beta$ in $(0,\alpha \wedge 1)$, it holds that
\begin{equation}
\label{eq:Estimates}
\begin{split}
    \sup_{t\in [0,T]} [u(t,\cdot) ]_{C^{\alpha+\beta}_d} \,&\le\,C_\beta \sup_{t\in (0,T)} [f(t,\cdot)]_{C^\beta_d};\\
    \text{ if $A=A_0$ in \eqref{eq:def_A_0}}, \qquad \quad [u]_{\dot W^{\alpha,p}_d (\R^N_T)} \, &\le \, C_p\, \Vert f\Vert_{L^p(\R^N_T)},
\end{split}
\end{equation}
where $C_\beta,C_p$ are the same constants as in \eqref{eq:Schauder_estim} and \eqref{eq:Sobolev_estim} respectively.
\end{theorem}

In the diffusive setting (i.e.\ when $\alpha=2$) the above result can be exploited to show that the estimates do not depend on the intensity of the matrix $B$ but only on the ellipticity constant $\kappa_2$  in \eqref{eq:def_kappa_2}. Indeed, under our hypothesis, it is possible to rewrite $B=\kappa_2 B_I+\tilde{B}$ where
\[B_I \, := \, \begin{pmatrix}
 I_{d_0}  & 0_{d_0,d_1}\\
 0_{d_1,d_0} & 0_{d_1,d_1}
\end{pmatrix}\]
and $\tilde{B}:=B-\kappa_2 B_I$ belongs to $\mathcal{S}_0(\R^N)$. Then, Theorem \ref{thm:main_theorem} implies that the constants appearing in \eqref{eq:Estimates} do not depend on $\|B\|$ but only linearly on $\kappa_2^{-1}$.

Once we will obtain the above result, it is easy to generalize it to more general operators. Indeed, let us consider the operator
\begin{equation}
\label{eq:def_operator_Lt}
\begin{split}
\mathscr{L}_{\alpha,t}\phi(x) \, := \, \mathcal{L}^{\text{pert}}_{\alpha,t}\phi(x) +\langle a(t),D_x\phi(x)\rangle -c(t)\phi(x),
\end{split}
\end{equation}
where $c\colon (0,T)\to [0,\infty)$, $a\colon (0,T)\to \R^N$ are two integrable functions and, we recall, $\mathcal{L}^{\text{pert}}_{\alpha,t} $ is defined in \eqref{eq:def_pert_operator} for a {bounded measurable} perturbation $S(t)$. For any sufficiently regular function $\phi\colon [0,T]\times\R^N\to \R$, we are going to denote
\begin{equation}
\label{eq:relation_solutions_time}
\mathcal{T}\phi(t,x) \, := \, e^{-\int_{0}^{t}c(s) \,ds}\phi\left(t,x+\int_{0}^{t}a(s) \,ds\right).
\end{equation}
It is not difficult to check that the ``operator'' $\mathcal{T}$ transforms solutions of the Cauchy Problem \eqref{eq:OU_initial:intro_pert} to solutions of the Cauchy Problem driven by $\mathscr{L}_{\alpha,t}$, even if for a modified source $\mathcal{T}f$.

\begin{corollary}
\label{prop:Schauder_estimate_time}
Under [\textbf{A}], let $f$ be in $B_b(0,T;C^\infty_c(\R^N))$. Then, there exists a unique solution $v$ to the Cauchy Problem
\begin{equation}
\label{eq:OU_initial:intro_gen}
\begin{cases}
    \partial_t v(t,x) \, = \,  \mathscr{L}_{\alpha,t} v(t,x) +f(t,x)\quad &\mbox{ on }\R^N_T;\\
   v(0,x)\, = \, 0 \quad &\mbox{ on } \R^N.
\end{cases}
\end{equation}
 Moreover, for any $p>1$ and any $\beta$ in $(0,1\wedge \alpha)$, it holds that
\begin{equation}\label{eq:Estimate_full_operator}
\begin{split}
    \sup_{t\in [0,T]} [v(t,\cdot)]_{C^{\alpha+\beta}_{d}(\R^N)} \,&\le\,C_\beta \sup_{t\in (0,T)} [f(t,\cdot)]_{C^\beta_d(\R^N)};\\
    \text{ if $A=A_0$ in \eqref{eq:def_A_0}}, \qquad [v]_{\dot W^{\alpha,p}_d (\R^N_T)} \, &\le \, C_p\, e^{\int_0^Tc(s)\, ds} \Vert f\Vert_{L^p(\R^N_T)},
\end{split}
\end{equation}
where $C_\beta$, $C_p$ are the same constants appearing in Theorem \ref{thm:main_theorem}.
\end{corollary}
\begin{proof}
We will use the following notation:
\[\tilde{c}(t) \, := \, \int_0^tc(s) \, ds, \quad \tilde{a}(t) \, := \, \int_0^ta(s) \, ds.\]

As explained before, it is not difficult to check that if $v$ is a solution to Cauchy Problem \eqref{eq:OU_initial:intro_gen}, then the function
\[u(t,x)\,:=\, \mathcal{T}^{-1}v(t,x) \, = \, e^{\tilde{c}(t)}v(t,x-\tilde{a}(t))\]
is the unique solution to \eqref{eq:OU_initial:intro_pert} with $\tilde{f}$ instead of $f$, where
\[\tilde{f}(t,x) \, := \, e^{\tilde{c}(t)} f(t,x-\tilde{a}(t)), \quad (t,x) \in \R^N_T.\]
Moreover, we have that $\tilde{f}$ is in $B_b\bigl(0,T;C^\infty_c(\R^N)\bigr)$.

Considering $t\le T$, we then notice from Theorem \ref{thm:main_theorem} (applied to $[0,t]\times \R^N$) that
\[  [u(t,\cdot)]_{C^{\alpha+\beta}_d(\R^N)}  \, \le \, C_\beta \sup_{s \in(0,t)} [\tilde{f}(s,\cdot)]_{C^\beta_d(\R^N)}.\]
Using now the invariance of the H\"older norm under translations, we can show that
\[[ v(t,\cdot)]_{C^{\alpha+\beta}_d(\R^N)} \, \le \, C_\beta e^{-\tilde{c}(t)}\sup_{s \in (0,t)} [ e^{\tilde{c}(s)}f(s,\cdot)]_{C^\beta_d(\R^N)} \,\le \, C_\beta \sup_{s \in(0,t)} [ f(s,\cdot)]_{C^\beta_d(\R^N)},\]
where in the last step we exploited that $\widetilde{c}(t)$ is non-decreasing. Taking the supremum with respect to $t$ on both sides of the above inequality, we obtain the first inequality in \eqref{eq:Estimate_full_operator}. For the second one,  notice from Theorem \ref{thm:main_theorem}, that
\begin{align*}
\sum_{i=0}^k\int_0^T e^{p\tilde{c}(t)} \Vert \Delta^{\alpha_i}_{x_i} v(t,\cdot)\Vert^p_{L^p(\R^N)} \, dt  &\le \,  C_p \int_0^T e^{p\tilde{c}(t)}\Vert f(t,\cdot)\Vert^p_{L^p(\R^N)} \, dt \, \\
&\le C_p \, e^{p\tilde{c}(T)} \int_0^T\Vert f(t,\cdot)\Vert^p_{L^p(\R^N)} \, dt.
\end{align*}
Using the fact that $e^{\tilde{c}(t)}\ge 1$ for all $t \in [0,T]$, we notice that
\[\sum_{i=0}^k\int_0^T \Vert \Delta^{\alpha_i}_{x_i} v(t,\cdot)\Vert^p_{L^p(\R^N)} \, dt \le C_p \,  e^{p\tilde{c}(T)} \int_0^T\Vert f(t,\cdot)\Vert^p_{L^p(\R^N)} \, dt,\]
and we have concluded.
\end{proof}

We now show how the results in Theorem \ref{thm:main_theorem} in the parabolic framework can be adapted to establish a-priori estimates for the elliptic one. More precisely, let us introduce
\begin{equation}\label{PERT_ELL}
\mathscr{L}_\alpha\phi(x) \, := \, \mathcal{L}^{\text{ou}}_\alpha\phi(x)+\langle a,D_x\phi(x)\rangle + \text{Tr}(S D^2_x \phi(x)),
\end{equation}
for a matrix $S$ in $\mathcal{S}_0(\R^N)$ and $a$ in $\R^N$. Let us mention that the Schauder and $L^p$ estimates are known for the \textit{unperturbed} operator in the diffusive case, i.e.\ the one obtained taking $\alpha=2 $ and $S=0$  in \eqref{PERT_ELL}. We can e.g.\ refer to the works \cite{Lunardi97}, \cite{Bramanti:Cupini:Lanconelli:Priola10} respectively. We strongly believe that under the current assumptions, the computations performed in \cite{Huang:Menozzi:Priola19} could be extended to derive such estimates when $\alpha\neq 2 $ and the matrix $A$ is invariant under dilations. Finally, we remark that when the matrix $A$ is invariant under dilations (i.e.\ $A=A_0$ as in \eqref{eq:def_A_0}), the constants $C_\beta$, $C_p$ in the estimates do not depend on the final time $T$.

\begin{corollary}
\label{coroll:Elliptic_res}
Under [\textbf{A}], let $g$ be in $C^\infty_c(\R^N)$ and $A=A_0$ as in \eqref{eq:def_A_0}. Assume that there exists a bounded, continuous (classic) solution $u\colon \R^N \to \R$ to the following elliptic equation:
\begin{equation}
\label{eq:Elliptic}
 \mathscr{L}_\alpha u(x)\, = \, g(x), \quad \text{on }\R^N.
\end{equation}
Then, such solution $u$ is unique and for any $p>1$ and any $\beta$ in $(0,1\wedge \alpha)$, it holds that
\begin{align*}
    [u]_{C^{2+\beta}_d(\R^N)} \,&\le\,C_\beta [g]_{C^\beta_d(\R^N)};\\
    [u]_{\dot W^{2,p}_d(\R^N)} \, &\le \, C_p\Vert g \Vert_{L^p(\R^N)},
\end{align*}
where $C_\beta$ and $C_p$ are the same constants appearing in Theorem \ref{thm:main_theorem}.
\end{corollary}
\begin{proof}
We start with the issue of uniqueness.  Let us consider a solution $u\colon \R^N \to \R$ to the elliptic Equation \eqref{eq:Elliptic}. Fixed a final time $T>0$, it is easy to check that the function $v\colon \R^N_T\to \R$ given by  $v(t,x):=u(x)t/T$ is then a solution of the following Cauchy Problem:
\[
\begin{cases}
\partial_tv(t,x)\, = \, \mathscr{L}_\alpha v(t,x) + f(t,x), &\mbox{ on } \R^N_T;\\
v(0,x) \, = \, 0, &\mbox{ on } \R^N,
\end{cases}
\]
where $f(t,x)=u(x)/T-g(x)t/T$. Thus, the uniqueness of $u$ follows immediately from the uniqueness of solutions in Theorem \ref{thm:main_theorem}. Moreover, since $A=A_0$, we know that $v$ satisfies Estimates \eqref{eq:Estimates} for constants $C_p,C_\beta$ independent from $T$. Hence, we have that
\[\begin{split}
[u]_{C^{\alpha+\beta}_d(\R^N)} \, &= \, \sup_{t\in [0,T]}[v(t,\cdot)]_{C^{\alpha+\beta}_d(\R^N)} \, \le \, C_\beta \sup_{t\in (0,T)}[f(t,\cdot)]_{C^{\beta}_d(\R^N)} \\
&\le \, C_\beta\left(\frac{[u]_{C^\beta_d(\R^N)}}{T}+[g]_{C^\beta_d(\R^N)}\right).
\end{split}\]
We can now let $T$ go to infinity in the equation above. Recalling that $C_\beta$ is independent from $T$, we conclude that
\[[u]_{C^{\alpha+\beta}_d(\R^N)} \, \le\,  C_\beta[g]_{C^\beta_d(\R^N)}.\]
In order to prove the Sobolev estimates, we firstly notice that for any $i\in \llbracket 0,k\rrbracket$,
\[
\Vert \Delta^{\alpha_i}_{x_i}v\Vert^p_{L^p(\R^N_T)} \, = \, \int^T_0\bigl(\frac{t}{T}\bigr)^p\int_{\R^N}|\Delta^{\alpha_i}_{x_i}u(z)|^p \, dzdt \, = \, \frac{T}{p+1}\Vert \Delta^{\alpha_i}_{x_i}u\Vert^p_{L^p(\R^N)}.\]
Using Estimates \eqref{eq:Estimates} for $v$, we then show that
\[
\begin{split}
\sum_{i=0}^d\Vert \Delta^{\alpha_i}_{x_i} u\Vert^p_{L^p(\R^N_T)} \, &= \,\frac{p+1}{T}  \sum_{i=0}^d  \Vert \Delta^{\alpha_i}_{x_i}v\Vert^p_{L^p(\R^N_T)} \, \le \, C^p_p\frac{p+1}{T}\Vert f \Vert^p_{L^p(\R^N_T)} \\
&\le \, C^p_p\frac{p+1}{T}\int_0^T\int_{\R^N}\left|\frac{u(z)}{T}-g(z)\frac{t}{T}\right|^p \, dzdt \\
&= \, C^p_p (p+1)\int_0^1\int_{\R^N}\left|\frac{u(z)}{T}-sg(z)\right|^p \, dzds,
\end{split}
\]
where, in the last step, we applied the change of variables $s=t/T$. Letting $T$ go to infinity, we finally notice that
\[ \sum_{i=0}^d    \Vert \Delta^{\alpha_i}_{x_i}u\Vert^p_{L^p(\R^N)} \, \le \, C^p_p\Vert g\Vert^p_{L^p(\R^N)} \int_0^1 (p+1)s^p \, ds \ = \,
C^p_p\Vert g \Vert^p_{L^p(\R^N)}.\]
We have thus concluded the proof of Corollary \ref{coroll:Elliptic_res}.
\end{proof}

\setcounter{equation}{0}
\section{Proof of the main result}
\label{SEC_PROOF}

Let $u\colon [0,T]\times \R^N\to \R$ be the solution to the initial (i.e.\ non-perturbed) Cauchy problem \eqref{eq:OU_initial:intro}. We can then introduce a function $v(t,x) := u(t,e^{-tA}x)$. Since $u$ is Lipschitz continuous in $t\in (0,T)$, it is possible to differentiate the function $u(t,x)=v(t,e^{tA}x)$ with respect to $t$, for almost every $t$ in $[0,T]$. It then follows that $v$ is the solution to
 \begin{equation} \label{ma}
 \begin{cases}
 \partial_tv(t, x)   \, = \, \tilde{\mathcal{L}}_{\alpha,t} v(t,x)  + \tilde f(t,x);\\
 v(0,x) =0,
 \end{cases}
\end{equation}
where we denoted $\tilde{f}(t,z):= f(t,e^{-tA}z)$ and
\begin{multline*}
\tilde{\mathcal{L}}_{\alpha,t} \phi(x) \, := \, \mathds{1}_{\{\alpha=2\}}\text{Tr} \left( e^{tA} B
 e^{tA^*} D^2_x v(t, x) \right)\\
 + \mathds{1}_{\{\alpha\neq2\}}\int_{\R^d}\left[v(t,x+e^{tA}\sigma z)-v(t,x)-\langle D_xv(t,x),e^{tA}\sigma z\rangle\mathds{1}_{B(0,1)}(\sigma z)\right] \nu_\alpha(dz).
 \end{multline*}
Noticing that $\det (e^{A_0t})=1$ so that $\Vert f\Vert_{L^p(\R^N_T)}=\Vert \tilde{f}\Vert_{L^p(\R^N_T)}$, the known estimates on $u$ in \eqref{eq:Max_princ}, \eqref{eq:Sobolev_estim} and \eqref{eq:Schauder_estim} can be rewritten in terms of the function $v$ as:
\begin{equation} \label{eq:Estimates_cv}
\begin{split}
\sup_{(t,x)\in[0,T]\times\R^N}|v(t,x)| \, &\le \, T \sup_{(t,x)\in\R^N_T}|\tilde{f}(t,x)|;\\
    \sup_{t\in [0,T]} [v(t,\cdot)]_{C^{\alpha+\beta}_{d,A}(\R^N)} \,&\le\,C_\beta \sup_{t\in (0,T)} [\tilde{f}(t,\cdot)]_{C^\beta_{d,A}(\R^N)};\\
    \text{ if $A=A_0$ in \eqref{eq:def_A_0}}, \qquad [v]_{\dot{W}^{\alpha,p}_{d,A} (\R^N_T)} \, &\le \, C_p\, \Vert \tilde{f}\Vert_{L^p(\R^N_T)},
\end{split}
\end{equation}
where we have denoted
\[
[v(t,\cdot)]_{C^{\alpha+\beta}_{d,A}(\R^N)} \, := \, [v(t,e^{At}\cdot)]_{C^{\alpha+\beta}_{d}(\R^N)}, \qquad [v]^p_{\dot{W}^{\alpha,p}_{d,A}(\R^N_T)} \, := \, \sum_{i=0}^k \|\Delta^{\alpha_i,A}_{x_i}v\|^p_{L^p(\R^N_T)}\]
and
\[
\Delta^{\alpha_i,A}_{x_i}v(t,x) \, := \, {\rm p.v.}\int_{\R^{\mathfrak d_i}}[v(t, x+ e^{tA} E_iz))-v(t, x)] \frac{dz}{|z|^{\mathfrak d_i+2\alpha_i}}.
\]

Clearly, it is possible to rewrite \eqref{ma} in the form of Cauchy problem \eqref{eq:Cauchy_problem_gen0} with $Q(t)=\mathds{1}_{\{\alpha=2\}}e^{tA}Be^{tA^\ast}$ and $\gamma(t)=\mathds{1}_{\{\alpha\neq 2\}}e^{tA}\sigma$, exploiting the symmetry of the measure $\nu_\alpha$.
By the arguments before and the maximum principle (Theorem \ref{thm:max_princ}, we already know that for any $\tilde{f}$ in $B_b(0,T;C^\infty_c(\R^N))$, there exists a unique solution $v$ to Cauchy problem \eqref{ma} and it satisfies the estimates in \eqref{eq:Estimates_cv}.

Let us consider again the {bounded measurable} perturbation $S\colon(0,T)\to \mathcal{S}_0(\R^N)$. We would like to show that there exists a unique solution $w\colon [0,T]\times\R^N\to \R$ to
\begin{equation} \label{ma0}
 \begin{cases}
 \partial_tw(t, x)   \, = \, \tilde{\mathcal{L}}_{\alpha,t} w(t,x) + \text{Tr} \left( e^{tA} S(t) e^{tA^*} D^2_x w(t, x) \right)  + \tilde f(t,x);\\
 w(0,x) =0,
 \end{cases}
\end{equation}
and that the estimates in \eqref{eq:Estimates_cv} hold for $w$ as well, with the same constants $C_\beta$, $C_p$ appearing before. Indeed, the backward change of variables argument explained at the beginning of the section will then allow us to conclude the proof of Theorem \ref{thm:main_theorem}. More precisely, we will prove at the end of the current section the following, more general, result:

\begin{prop}
\label{prop:pass_to_lim}
Let $\tilde{f}$ be in $B_b(0,T;C^\infty_c(\R^N))$ and $\tilde{S}\colon(0,T)\to\mathcal{S}_0(\R^N)$ {bounded measurable}. Then, there exists a unique solution $w$ to:
\begin{equation} \label{eq:PDE_pass_limit}
\begin{cases}
\partial_tw(t, x)   \, = \, \tilde{\mathcal{L}}_{\alpha,t} w(t,x) + {\rm Tr} \left( \tilde{S}(t)D^2_x w(t, x) \right)+\tilde{f}(t,x),\\
w(0,x) =0.
\end{cases}
\end{equation}
Moreover, the estimates in \eqref{eq:Estimates_cv} hold again with $w$ instead of $v$ and with the same constants appearing before.
\end{prop}

It is clear that Theorem \ref{thm:main_theorem} will then follow from the previous result by choosing $\tilde{S}(t):=  e^{tA} S(t) e^{tA^*}$.
The crucial observation in the proof of the above result is that, at least formally, we can approximate the diffusive perturbation $\tilde{S}(t)$ as follows:
\begin{equation}
\label{eq:approx_fin_diff}
\begin{split}
\frac{1}{2}\text{Tr}\left(\tilde{S}(t)D^2_x\phi(x)\right) \, &\approx \, \frac{1}{2}\sum_{i=1}^N\frac{1}{\varepsilon^2}\left[\phi(x+\varepsilon \sqrt{\tilde{S}(t)}e_i)-2\phi(x)+\phi(x-\varepsilon \sqrt{\tilde{S}(t)}e_i)\right] \\
&= \, \sum_{i=1}^N\lambda_{i}\left(\delta_{l_{i}(t)}-\delta_{-l_{i}(t)}\right)\phi(x) \, =: \, J^{\tilde{S}}_{t,\varepsilon}\phi(x),
\end{split}
\end{equation}
where $\lambda_{i}=\frac{1}{2}\varepsilon^{-2}$, $l_{i}(t)=\varepsilon\sqrt{\tilde{S}(t)}e_i$ and $\delta_y\phi(\cdot):=\phi(\cdot+y)-\phi(\cdot)$ is the first order finite difference operator.

In order to apply the above heuristic argument in our case, when $\tilde{S}(t)$ is only Borel measurable, we firstly need to show the existence of a Borel measurable principal square root for $\tilde{S}(t)$. While this fact is well-known if $\tilde{S}(t)$ is positive-definite, a proof of the following result could be found in \cite[Result $5.4$]{Reid70}\footnote{
For the sake of completeness of the work, we also highlight that the measurability of the principal square root of the matrix $\tilde{S}(t)$ could be obtained from the following integral representation
\[\sqrt{\tilde{S}(t)}\,= \, c\int_0^{+\infty}\theta^{-3/2}(I-\exp(-\theta \tilde{S}(t)))\, d\theta,\]
where $c>0$ is a normalising constant, noticing that the exponential of a measurable matrix is still measurable.
}.

\begin{lemma}
\label{lemma:measur_square_root}
Let $Q\colon (0,T)\to \mathcal{S}_0(\R^N)$ be a measurable function. Then, there exists a measurable function $\sqrt{Q}\colon (0,T)\to \mathcal{S}_0(\R^N)$ such that $\sqrt{Q}(t)$ is a square root of $Q(t)$ for any fixed $t$ in $(0,T)$.
\end{lemma}

As a first step in our method of proof, we will show that it is possible to add a finite difference operator, like the one appearing in \eqref{eq:approx_fin_diff}, to the Cauchy problem without modifying the estimates. This is indeed the key point in the proof method of \cite{Krylov:Priola17} and it is performed through a probabilistic argument which relies on the use of a suitable corresponding Poisson process. More precisely, we have the following result:

\begin{lemma}
\label{lemma:finite_diff}
Let $\lambda$ be a real number, $l\colon (0,T)\to \R^N$ a {bounded measurable} function and $\tilde{f}$ in $B_b(0,T;C^\infty_c(\R^N))$. Then, there exists a unique solution $w$ to
\begin{equation} \label{ma3} \begin{cases} \partial_tw(t, x)   \, = \, \tilde{\mathcal{L}}_{\alpha,t} w(t,x)  + \lambda \delta_{l(t)}w(t,x)+\tilde f(t,x);\\
 w(0,x) \, = \, 0,
 \end{cases}
\end{equation}
Moreover, the estimates in \eqref{eq:Estimates_cv} hold for $w$ as well, with the \emph{same} constants appearing therein.
\end{lemma}
\begin{proof}
Let us assume for the moment that $\lambda>0$. We can define
\[ X_t = \int_{0}^t l(r)\, d \pi_r\]
where  $\{\pi_t\}_{t\ge0}$ is a scalar Poisson process of intensity $\lambda$ defined on a probability space $\left(\Omega,\mathcal{F},\mathbb{P}\right)$.
We notice that the random shifted source $\bar{f}(t,x):=\tilde{f}(t,x-X_t)$ is again in $B_b\left(0,T;C^\infty_0(\R^N)\right)$. We have omitted to write the dependence on the random parameter $\omega$ for notational simplicity. We again emphasise that, from the very nature of the Poisson process involved, the modified (random) source $\bar f $ has jumps in time.
Then, for (almost every) fixed $\omega$ in $\Omega$, we already know that there exists a unique solution $v$ to \eqref{ma}, replacing therein $\tilde f$ with the random source $\bar{f}$, depending on  $\omega$ as parameter. Moreover, thanks to the invariance for translations of the involved norms, it follows from \eqref{eq:Estimates_cv} that
\begin{equation}\label{w11}
\begin{split}
   \sup_{(t,x)\in[0,T]\times\R^N}|v(t,x)| \, &\le \, T \sup_{(t,x)\in\R^N_T}|\tilde{f}(t,x)|; \\
    \sup_{t\in [0,T]} [v(t,\cdot)]_{C^{\alpha+\beta}_{d,A}(\R^N)} \, &\le\,C_\beta \sup_{t\in (0,T)} [\tilde{f}(t,\cdot)]_{C^\beta_{d,A}(\R^N)};\\
    \text{ if $A=A_0$}, \qquad [v]_{\dot{W}^{\alpha,p}_{d,A}(\R^N_T)} \, &\le  \, C_p \|   \tilde{f}  \|_{L^p(\R^N_T)}.
\end{split}
\end{equation}
It is not difficult to check (cf.\ \cite{book:Applebaum09}) that the solution $v$ is given by
\[
v(t,x) \, = \, \int_0^t\int_{\R^N} \tilde{f}(s,x - X_s + y) \, \mu_{ {s,t}} (dy)ds,
\]
where $\mu_{s,t}$ is the law of the stochastic integral
\[I_{s,t} \, := \, \int_s^t e^{rA}\sigma\, dZ_r,\]
with $\{Z_t\}_{t\ge 0}$ a Brownian motion if $\alpha=2$ or an $\alpha$-stable process with L\'evy measure $\nu_\alpha$, otherwise.

For each $x \in \R^N$, the stochastic process  $ (v(t,x))_{t \in [0,T]} $ has continuous paths ($\P$-a.s.) and it is ${\cal F}_t$-adapted where  ${\cal F}_t$ is the completed  $\sigma$-algebra generated by the random variables $\pi_s$, $0 \le s \le t$. For fixed $x \in \R^N$, let us introduce the process $(v(t,x+ X_t))_{t \in [0,T]}$ which is given by
\[
 v(t,x+ X_t) = \int_0^t\int_{\R^N} [\tilde{f}(s,x + X_t -  X_s + y) \mu_{ {s,t}} (dy)ds.
\]
It is not difficult to check that it is ${\cal F}_t$-adapted and   it has  c\`adl\`ag paths.

Denoting by $\sigma_n, \ n\in \{0,\dots, \pi_t\} $ the jump times of the process $(X_s)_{s\in [0,t]}$ and setting as well $\sigma_0=0 $, we have that $X_s$ is constant for $s\in[\sigma_n,\sigma_{n+1}\wedge t )$.  One then derives that:
\begin{align}
\label{eq:1_OU}
v(t,x+ X_t&)  \,= \,v(t,x+ X_t)-v(0,x)\notag\\
&= \,\sum_{n=0}^{\pi_t-1} \Bigl[v(\sigma_{n+1}^-,x+X_{\sigma_{n+1}^-})-v(\sigma_n,x+X_{\sigma_n})+v(t,x+X_t)\notag\\
&\qquad-v(\sigma_{\pi_t},x+X_{\sigma_{\pi_t}})+v(\sigma_{n+1},x+X_{\sigma_{n+1}})-v(\sigma_{n+1}^-,x+X_{\sigma_{n+1}^-}) \Bigr]\notag\\
&= \, \int_0^t\left[\tilde{\mathcal{L}}_{\alpha,t}v(s,x+ X_s))+ \tilde f(s,x)\right]\, ds +\int_0^t g(s,x) \,d\pi_s,
\end{align}
where $g(s,x)=v(s,x+ l(s) + X_{s-})-v(s,x+ X_{s-})$ is precisely the contribution associated with the jump times. Here, $X_{s-}=\lim_{t\uparrow s}X_{t}$, $s>0$.
It is clear that $g(s,x)\neq 0$ if and only if $\pi_s$ has a jump at time $s$.
We then have:
\begin{equation}
\label{proof:eq1}
    \mathbb{E}\left[\int_0^tg(s,x)\, d\pi_s\right]  \, = \, \lambda \int_0^t \left[w(s,x+l(s))-w(s,x)\right]\, ds,
\end{equation}
where $w(s,x) = \mathbb{E}[v(s,x+X_s)]$. The above equality can be easily proven when $g$ is piece-wise constant and then extended to our more general context by standard approximation arguments (cf.\ Lemma $2.1$ in \cite{Krylov:Priola17}). Taking the expectation on both sides of Equation \eqref{eq:1_OU}, we find out that $w$ is an integral solution to \eqref{ma3}. Moreover by \eqref{w11} we obtain (using also the Jensen inequality and  the Fubini theorem)
\[
\begin{split}
[w]_{\dot{W}^{\alpha,p}_{d,A}(\R^N_T)}^{p} \, &= \, \sum_{i=0}^k \int_{\R^N_T} |\Delta^{\alpha_i,A}_{x_i}w(t,x)|^p \, dx dt \, = \, \sum_{i=0}^k \int_{\R^N_T} \left|\mathbb{E}\left[\Delta^{\alpha_i,A}_{x_i}v(t,x + X_t)\right]\right|^p \, dx dt \\
&\le \, \sum_{i=0}^k   \int_{\R^N_T} \E \left[ \left|\Delta^{\alpha_i,A}_{x_i}v(t,x + X_t)\right|^p\right] \, dx dt\\
&= \, \sum_{i=0}^k\E\left[\int_{\R^N_T} \left|\Delta^{\alpha_i,A}_{x_i}v(t,x + X_t)  \right|^p dx dt\right]\\
&=\, \sum_{i=0}^k {\mathbb E}\left[\int_{\R^N_T}  |\Delta^{\alpha_i,A}_{x_i}v(t,x') |^p\, dx' dt\right]\,
\le \, (C_p)^p  \|   \tilde{f}\|_{L^p(\R^N_T)}^p.
\end{split}
\]
A similar argument works as well for the Schauder estimates and the maximum principle. Uniqueness of solutions to \eqref{ma3} follows by Theorem \ref{thm:max_princ}. Indeed, if we assume for the moment that $\lambda T \le 1/4$, a solution $w$ to \eqref{ma3} is also a solution to Cauchy problem \eqref{ma} with source $\tilde{f}(t,x)+\lambda[w(t,x+l(t))-w(t,x)]$. Since $\tilde{f}$ is regular enough, we can then apply Theorem \ref{thm:max_princ} and, by a circular argument, conclude that
\begin{align*}
\sup_{(t,x)\in[0,T]\times\R^N}|w(t,x)| \, \le& \, T \Big(\sup_{(t,x)\in\R^N_T}|\tilde{f}(t,x)|+2\lambda\sup_{(t,x)\in[0,T]\times\R^N} |w(t,x)|\Big) \, \\
\le & \, 2T \sup_{(t,x)\in\R^N_T}|\tilde{f}(t,x)|,
\end{align*}
when precisely $\lambda T \le 1/4$. The above estimates clearly imply the uniqueness of solutions. By iteration of the above procedure by steps of size $1/(4 \lambda)$, we can claim the uniqueness of solutions over the full interval $[0,T]$.

By considering the opposite random shift $\bar{f}(t,x):=\tilde{f}(t,x+X_t)$, the same arguments above allows to obtain the result when $\lambda<0$.
\end{proof}

Now, by an easy iterative argument using the above Lemmas \ref{lemma:finite_diff} and \ref{lemma:measur_square_root}, it is not difficult to show that

\begin{prop}
\label{prop:fin_diff_perturb}
Fixed $\varepsilon>0$, let $\tilde{f}$ be in $B_b(0,T;C^\infty_c(\R^N))$. Then, there exists a unique solution $w_\varepsilon$ to
\begin{equation}\label{eq:fin_diff_perturb}
\begin{cases}
\partial_t w_\varepsilon(t,x) \, = \, \tilde{\mathcal{L}}_{\alpha,t} w_\varepsilon(t,x) + J^{\tilde{S}}_{t,\varepsilon}w_\varepsilon(t,x)+\tilde{f}(t,x);\\
w_\varepsilon(0,x) \, = \, 0,
\end{cases}
\end{equation}
where, we recall, the operator $J^{\tilde{S}}_{t,\varepsilon}$ has been defined in \eqref{eq:approx_fin_diff}. Moreover, the estimates in \eqref{eq:Estimates_cv} hold again with $w_\varepsilon$ instead of $v$ and with the same constants appearing before.
\end{prop}

In order to conclude and show Proposition \ref{prop:pass_to_lim}, we are interested now in understanding what happens when we let $\varepsilon$ goes to zero in \eqref{eq:fin_diff_perturb}. In particular, we want the corresponding second order term in \eqref{eq:approx_fin_diff} to appear at the limit.

\begin{proof}[Proof of Proposition \ref{prop:pass_to_lim}]
Let us recall from Proposition \ref{prop:fin_diff_perturb} that there exists a family $\{w_\varepsilon\}_{\varepsilon>0}$ such that $w_\varepsilon\colon [0,T]\times\R^N\to \R$ is the unique solution to \eqref{eq:fin_diff_perturb}. Moreover, it holds that
\begin{equation}
\label{eq:proof_max_princ}
\sup_{(t,x)\in[0,T]\times\R^N}|w_\varepsilon(t,x)| \, \le \, T \sup_{(t,x)\in\R^N_T}|f(t,x)|.
\end{equation}
For a multi-index $\gamma$ in $\N^N$ and $h>0$, let us denote by $\Delta^\gamma_h$ the iteration of the finite difference operators
\[\Delta_h^i\phi(x) \, = \, \frac{\phi(x+he_i)-\phi(x)}{h}, \quad h>0, \, i\in \llbracket 1,N\rrbracket.\]
We then notice that $\Delta_h^\gamma w_\varepsilon(t,x)$ solves again \eqref{eq:fin_diff_perturb} with $\Delta_h^\gamma f$. Using that $f$ is in $B_b(0,T;C^\infty_c(\R^N))$ and the maximum principle in \eqref{eq:proof_max_princ}, it then follows that $w_\varepsilon$ is smooth in space and $D^\gamma_xw_\varepsilon$ is bounded in $t,x,\varepsilon$. Equation \eqref{eq:fin_diff_perturb}, to be understood in its integral form, together with \eqref{eq:proof_max_princ} then gives that those derivatives are themselves Lipschitz continuous in time, uniformly in $x,\varepsilon$.  This precisely gives that the family $\{w_\varepsilon\colon \varepsilon>0\}$ is equi-Lipschitz continuous on any compact subset $K$ of $[0,T]\times\R^N$. Similar arguments hold as well for the spatial derivatives of $w_{\varepsilon}$.

We can now apply the Arzel\`a-Ascoli theorem to $w_\varepsilon$ showing the existence of a sub-sequence $\{w_{\varepsilon_n}\}_{n \in \N}$ which converges uniformly on any compact set to a function $w\colon [0,T]\times \R^N\to \R$. Moreover, $w$ is smooth in space and all its space derivatives are Lipschitz continuous on any compact set. Similarly, any  derivative in space of $w_{\varepsilon_n}$ tends to the respective derivative of $w$, uniformly on the compact sets. Passing to the limit as $n\to \infty$ along the sub-sequence $\{\varepsilon_n\}_{n\in \N}$ in the estimates in \eqref{eq:Estimates_cv}, we immediately find out that they hold again with $w$ instead of $v$ and with the same constants appearing before. Letting $\varepsilon_n$ goes to zero in Cauchy problem \eqref{eq:fin_diff_perturb}  (written in the  integral form), we can also conclude that $w$ solves
\[  \partial_t w(t,x) \, = \, \tilde{\mathcal{L}}_{\alpha,t} w(t,x) + {\rm Tr} \left(\tilde{S}(t)D^2_x w(t, x) \right) +\tilde{f}(t,x).
\]
Indeed, the dominated convergence theorem immediately implies that
\[\lim_{n\to +\infty}\int_0^t J^{\tilde{S}}_{t,\varepsilon_n} w_{\varepsilon_n}(s,x)\, ds \, = \, \int_{0}^t\left[{\rm{Tr}}\left( \tilde{S}(s)D^2_xw(s,x)\right)\right] \, ds.\]
Similarly, it also holds that
\[\lim_{n\to +\infty}\int_{0}^t\left[{\rm{Tr}}\left( BD^2_xw_{\varepsilon_n}(s,x)\right)\right] \, ds \, = \, \int_{0}^t\left[{\rm{Tr}}\left( BD^2_xw(s,x)\right)\right] \, ds,\]
when $\alpha=2$. On the other hand (i.e.\ $\alpha\neq 2$), we have that
\begin{multline*}
    \int_0^t \tilde{\mathcal{L}}_{\alpha,t}w_{\varepsilon_n}(s,x)\, ds \, = \, \int_0^t\int_{|\sigma z|\ge 1}\left[w_{\varepsilon}(s,x+e^{sA}\sigma z)-w_{\varepsilon}(s,x)\right] \nu_\alpha(dz)\\
    +\int_0^t\int_{|\sigma z|<1}\left[w_{\varepsilon}(s,x+e^{sA}\sigma z)-w_{\varepsilon}(s,x)-\langle D_xw_{\varepsilon}(s,x),e^{sA}\sigma z\rangle\right] \nu_\alpha(dz).
\end{multline*}
While the first integral is clearly convergent by the bounded convergence theorem, we notice that
\[|w_{\varepsilon}(s,x+e^{sA}\sigma z)-w_{\varepsilon}(s,x)-\langle D_xw_{\varepsilon}(s,x),e^{tA}\sigma z\rangle| \, \le \, |e^{sA}\sigma z|^2\Vert D^2_xw_{\varepsilon}\Vert_\infty\]
and the dominated convergence theorem can thus be applied to the second integral, thanks to a uniform estimate for $D^2_xw_\varepsilon$ (cf.\ \eqref{eq:proof_max_princ}). Finally, the uniqueness of such solution $w$ follows immediately from the maximum principle (Theorem \ref{thm:max_princ}).
\end{proof}

\bibliography{bibli}
\bibliographystyle{abbrv}
\end{document}